\documentclass[journal,onecolumn]{IEEEtran}
\usepackage[dvips]{graphicx}
\usepackage[backend=biber,maxnames=5]{biblatex}
\usepackage{hyperref} 
\usepackage{verbatim}   
\usepackage{xcolor,colortbl}
\usepackage{url}
\usepackage[cmex10]{amsmath}
\usepackage{amssymb}
\usepackage{amsthm}
\usepackage{ifthen}
\usepackage{todonotes}
\usepackage{pgf}
\usepackage{tikz}
\usepackage{epsfig}
\usetikzlibrary{patterns}

\title{Modular Golomb Rulers and Almost Difference Sets}
\author{Daniel M. Gordon
  \thanks{D.M. Gordon is with the IDA Center for Communications
    Research-La Jolla, 4320 Westerra Court, San Diego, CA 92121, USA (email: gordon@ccr-lajolla.org)}
}

\begin{document}

\maketitle

\newtheorem{question}{Question} 
\newtheorem{idea}{Idea} 

\newtheorem{theorem}{Theorem}
\newtheorem{lemma}[theorem]{Lemma}
\newtheorem*{conjecture}{Conjecture}
\newtheorem{remark}{Remark}
\newtheorem{claim}{Claim}
\newtheorem{corollary}[theorem]{Corollary}
\newtheorem{proposition}[theorem]{Proposition}
\newtheorem{definition}[theorem]{Definition}

\newcommand{\mgr}[1]{\mathop{{\mathsf{MGR}}({#1})}}
\newcommand{\hatt}{\mbox{$\hat{t}$}}
\def\Znum{{\mathbb{Z}}} 
\def\Qnum{{\mathbb{Q}}} 

\newcommand{\dnd}{\;\!\mid\mskip -9mu \not \;\;}


\setcounter{table}{0}  
\setcounter{figure}{0}  

\begin{abstract}
A $(v,k,\lambda)$-difference set in a group $G$ of order $v$ is a subset
$\{d_1, d_2, \ldots,d_k\}$ of $G$ such that $D=\sum d_i$ in the group
ring $\Znum[G]$ satisfies 
$$D D^{-1} = n + \lambda G,$$
where
$n=k-\lambda$.  In other words, the nonzero elements of $G$ all occur
exactly $\lambda$ times as differences of elements in $D$.

A $(v,k,\lambda,t)$-almost difference set has $t$ nonzero elements of
$G$ occurring $\lambda$ times, and the other $v-1-t$ occurring
$\lambda+1$ times.  When $\lambda=0$, this is equivalent to a modular
Golomb ruler.
In this paper we investigate existence questions on these objects, and
extend previous results constructing almost difference sets by adding or
removing an element from a difference set.  We also
show for which primes the octic residues, with or without zero, form
an almost difference set.
\end{abstract}

\maketitle

\section{Introduction}

Let $G$ be a finite group $G$ of order $v$, and
$\Znum[G]=\{\sum_{g\in G}a_gg\mid a_g\in \Znum\}$ denote the group ring of $G$
over $\Znum$.  
We will identify a set $D=\{d_1, d_2, \ldots,d_k\} \subset G$ with its
group ring element $\sum_{i=1}^k d_i$.
A $k$-subset $D$ of $G$ is a 
$(v,k,\lambda)$-difference set if
\begin{equation}\label{eq:ds}
D D^{-1} = n + \lambda G,  
\end{equation}
where $n=k-\lambda$.
There is a large literature on difference sets.  See, for example,
\cite{bjl} or  \cite{crcds}.

Let $G$ be a group of order $mn$, and $N$ be a normal subgroup of $G$
of order $n$.
A relative $(m,n,k,\lambda)$-difference set of $G$ relative to $N$ is
a $k$-subset $D$ of $G$ such that
$$
D D^{-1} = n + \lambda (G-N).
$$
$N$ is referred to as the {\em forbidden subgroup}.
We will need
relative difference sets in Section~\ref{sec:mgr}; see \cite{pott1996survey} for background.

In a $(v,k,\lambda)$-difference set, every nonzero element $g \in G$ occurs exactly $\lambda$
times as the difference of two elements of $D$.  If some (say $t$) of
the $g$ occur
$\lambda$ times, and the rest $\lambda+1$, then $D$ is called a
$(v,k,\lambda,t)$-{\em
  almost difference set} (ADS).  
The complement of a $(v,k,\lambda,t)$-ADS is a
$(v,v-k,v-2k+\lambda,t)$-ADS.

This definition of ADS, given in \cite{dhm2001},
generalizes two earlier definitions: Ding in
\cite{ding95}, \cite{ding1997cryptographic} specified $t=(v-1)/2$, so
that exactly half of the differences are $\lambda$.
Davis, in \cite{davis1992almost}, defined ADS to have either the
$\lambda$ or $\lambda+1$ differences being all the nonidentity
elements of a normal subgroup $N$ of $G$ (equivalently, a
$(m,n,k,\lambda,\lambda \pm 1)$-divisible difference set).

Almost difference sets have many applications in communications and
cryptography.
They have been used to construct codebooks and sequences with good
properties for CDMA systems, for example see \cite{ding2008codebooks}
and \cite{li2019more}.
Constructions of good linear error-correcting codes from ADS are given
in \cite{ding} and \cite{heng2022projective}.
ADS may be used to construct cryptographic
functions with good nonlinearity, see Chapter 6 of \cite{cusick2004stream}.

A {\em Golomb ruler} of order $k$ is a set of $k$ distinct integers
$d_1 <  d_2< \ldots < d_k$ for which all the differences $d_i - d_j$
($i \neq j$)
are distinct. The {\em length} of a Golomb ruler is $d_k - d_1$.
A $(v,k,0,t)$-ADS is also known as a $(v,k)$-{\em modular Golomb
  ruler} (MGR),
a $k$-element set for which the differences are a packing of the integers
modulo $v$.
MGRs have been studied separately from general ADS; see \cite{bs2021}
and \cite{bs2022}
for recent results.  They have applications to
optical orthogonal codes \cite{yu2013deterministic}.
fingerprint designs \cite{zhao2015design},
and locally repairable codes \cite{choi2025optimal}.

Arasu et al. \cite{adhkm2001} gave the following theorems for constructing
ADS by adding or removing an element from particular difference sets.

\begin{theorem}\label{thm_remove}
  Let $D$ be an $(v,(v+3)/4,(v+3)/16)$-difference set of $G$, and let
  $d$ be an element of $D$.  If $2d$ cannot be written as the sum of
  two distinct elements of $D$, then $D \backslash \{d\}$ is an
  $(v,(v-1)/4,(v-13)/16,(v-1)/2)$-almost difference set of $G$.
\end{theorem}

\begin{theorem}\label{thm_add}
  Let $D$ be an $(v,(v-1)/4,(v-5)/16)$-difference set of $G$, and let
  $d$ be an element of $G \backslash D$.  If $2d$ cannot be written as the sum of
  two distinct elements of $D$, then $D \cup \{d\}$ is an
  $(v,(v+3)/4,(v-5)/16,(v-1)/2)$-almost difference set of $G$.
\end{theorem}

Daza et al. \cite{urbano2021almost} generalized these theorems:

\begin{theorem}\label{thm:gen_add}
  Let $D$ be a $(v,k,\lambda)$-difference set in $G$.  If
  \begin{enumerate}
  \item $g \in G \backslash D$,
  \item $(g-D) \cap (D-g) = \emptyset$,
  \end{enumerate}
  then $D\cup \{g\}$ is a $(v,k+1,\lambda,v-1-2k)$-ADS.  
\end{theorem}

\begin{theorem}\label{thm:gen_remove}
  Let $D$ be a $(v,k,\lambda)$-difference set in $G$.  If
  \begin{enumerate}
  \item $d \in D$,
  \item $(d-D) \cap (D-d) = \{0\}$,
  \end{enumerate}
  then $D\backslash \{d\}$ is
  a $(v,k-1,\lambda-1,2(k-1))$-ADS.
\end{theorem}

Using this they show that an element may be removed from
a $(v,k,1)$-planar difference set (referred to there as a Singer type
Golomb ruler) to get
$(v,k-\ell,0,\ell(2k-l-1)$-ADS  by removing any $\ell$ elements from $D$, for
any $\ell=1,2,\ldots,k$.  By adding an element they get a
$(v,k+1,1,v-1-2k)$-ADS 

Note that planar difference sets are only known to exist for $k-1$ a
prime power.  It has been shown by Peluse \cite{peluse2021asymptotic}
that asymptotically almost all planar difference sets are of this
form, and no counterexamples exist with $k < 2\cdot 10^{10}$
\cite{gordon2022difference}.



In the next section we will investigate other difference sets where an
ADS may be formed using one of these constructions.
Section~\ref{sec:mgr} looks at the special case of modular Golomb
rulers, and Section~\ref{sec:comp} gives the results of some
computational searches.  In the appendix we show for which primes the octic
residues, either with or without zero added, form an almost difference set.

\section{Adding or removing elements from Difference sets}\label{sec:addremove}

Call the set of all sums of two distinct elements of $D$ the 
{\em sumset} of $D$, and denote it by $S(D)$:
$$
S(D) = \left\{ d_i + d_j: d_i \neq d_j \in D \right\}
$$

The following is a restatement of Theorems~\ref{thm:gen_add} and
\ref{thm:gen_remove}, which will be used to show that other difference
sets form almost difference sets by adding or removing an element.
  
\begin{theorem}\label{thm:sumset}
  Let $D$ be a $(v,k,\lambda)$-difference set in $G$ with $v$ odd.  If
  $d \in D$ and $2d\not\in S(D)$, then $D\backslash \{d\}$ is
  a $(v,k-1,\lambda-1)$-ADS.  If $g \not\in D$ and $2g \not\in S(D)$,
  then $D\cup \{d\}$ is a $(v,k+1,\lambda)$-ADS.  
\end{theorem}

For this version, the condition $v$ odd is necessary.  For example,
the sumset of the $(16,6,2)$
difference set in $G = \Znum_2 \times \Znum_8$:
$$
\{ (0, 0), (0, 1), (0, 2), (0, 5), (1, 0), (1, 6) \}
$$
does not contain $(0, 0)$, $(0, 4)$, or $(1, 4)$, but $(1,4)$ is not
$2d$ for any $d \in G$.  Removing $(0,0)$ does not result in an ADS,
since $(0,0) - (1,0) = (1,0) - (0,0)$, and the same holds for removing
$(0,2)$.

All of the theorems in the last section have $v$ odd.  The fact that
an element can be added or removed from a planar difference set
follows from the fact that 
  the size of $S(D)$ is at most $\binom{k}{2}$.  If
  $v-k > |S(D)|$, then some element $d$ has $2d \not\in S(D)$ and $d
  \not\in D$.
  We have $\lambda(v-1) = v-1 = k (k-1)$, so
  $v -k = (k-1)^2$.  For $k \geq 3$
  $(k-1)^2 > \binom{k}{2}$.

The difference sets in
Theorems~\ref{thm_remove} and \ref{thm_add} are of type
$B_0$ and $B$ (respectively) in Hall's notation \cite{hall1998combinatorial}; quartic residues
with and without zero.  Both theorems depend on $0 \not\in S(D)$.

\begin{lemma}
  For $v$ as in Theorems~\ref{thm_remove} and \ref{thm_add}, no
  pair of quartic residues sum to zero.
\end{lemma}

\begin{proof}
  In both Theorems, $v \equiv 5 \pmod 8$, so $-1$ is not a
  quartic residue (this follows immediately from a generalization
  of Euler's criterion; see, for example, Theorem 2.8 of \cite{rose1995course}).
  Thus the negative of a  quartic residue is a nonresidue, and so not in $D$.
\end{proof}

Thus $d=0$ may be used for either theorem.  This was implicit in
\cite{adhkm2001}, since Theorem 2 of that paper points out that
quartic residues with or without $\{0\}$ in the two cases are almost
difference sets.

Similar results hold for other residue difference sets.  For quadratic
residues (Paley difference sets), the resulting ADS is actually a
difference set:

\begin{theorem}
  Let $v=4n-1$ be prime, and $D$ the $(4n-1,2n-1,n-1)$ Paley difference set in 
  ${\rm GF}(v)$.  Then $D \cup \{0\}$ is a
  $(4n-1,2n,n-1,v-1-2k)$-difference set.
\end{theorem}

\begin{proof}
  $D$ consists of all the squares in ${\rm GF}(v)^*$.  
  Each of the
  nonzero elements of ${\rm GF}(v)$ can be expressed in exactly
  $\lambda$ ways as differences of $D$.  Since
 $-1$ is a nonsquare mod $v$, so
  the new differences $\{d\}$
  and $\{-d\}$ for $d \in D$ are all distinct, the former being
  squares and the latter nonsquares.

  In fact, this is equivalent to the complement of the original difference set.
\end{proof}

Marshall Hall \cite{hall1998combinatorial} defines a
type $O$ difference set to be the octic residues of a prime
$p = 8 a^2+1 = 64 b^2+9$, with $a, b$ odd, and type 
$O_0$ difference set to be $\{0\}$ and the octic residues of a prime
$p = 8 a^2+49 = 64 b^2+441$, with $a$ odd, $b$ even.

\begin{theorem}\label{thm_octic}
  Let $D$ be a difference set of type $O$ (respectively type $O_0$).
  Then adding (respectively removing) $\{0\}$ from $D$ gives an almost
  difference set.
\end{theorem}

The proof is the same as before; in both cases $p \equiv 9 \pmod
{16}$, so $-1$ is an octic
nonresidue (again from the generalized Euler criterion),
and no two distinct elements of $D$ will sum to zero, so
adding or removing zero changes the number of times a difference
appears by 0 or 1.

The type $O$ almost difference set was pointed out in Theorem 2 of 
\cite{adhkm2001}.  
The type $O_0$ ADS was shown in \cite{nowak2014survey}, which has not
been published.
The smallest type $O$ ADS given by Theorem~\ref{thm_octic} is
$(26041,3255,406,6510)$.
The smallest type $O_0$ almost difference set is $(73,10,1,54)$, followed by
$(104411704393,13051463050,1631432881,78308778294)$.

In \cite{ding1997cryptographic} it is shown that 
$(v,k,\lambda,t)$-almost
difference sets of type $O$ exist only for $v=41$ or primes $v=8f+1$, where $f \equiv 5
\pmod 8$, 2 is a quartic residue modulo $v$, and $v$ has the
representations
$$
v = 19^2 + 4 y^2 = 1 + 2 b^2
$$
or
$$
v = 13^2 + 4 y^2 = 1 + 2 b^2.
$$
However, this paper was using Ding's original definition of ADS, where
$t=(v-1)/2$.
This theorem has appeared numerous times in the literature
(\cite{cusick2004stream},
\cite{ding},
\cite{ding1999several},
\cite{qi2016nonexistence}), not always making it clear which
definition of ADS it applies to, and the corresponding theorem for
general ADS has not appeared.
Theorem~\ref{thm_octic} shows that, for the now-standard definition, there are other octic ADS.
In the appendix we determine all such ADS of type $O$ and $O_0$.

Theorems~\ref{thm:gen_add} and \ref{thm:gen_remove} are applied in
\cite{urbano2021almost} 
to planar difference sets.  They (and Theorem~\ref{thm:sumset} for $v$ odd)
may also be applied to 
$(v,k,2)$-difference sets, known as {\em biplanes}.  There are still
$\binom{k}{2}$  sums of distinct elements of $D$, and
$$
\lambda (v-1) = 2(v-1) = k(k-1) = 2 \binom{k}{2}
$$
so $v = \binom{k}{2}+1$, so there is at least one element mod $v$ not
in $S(D)$.

Unfortunately, 
very few biplanes are known.  Aside from the trivial
$(2,2,2)$ and $(4,3,2)$, and $(7,4,2)$ (complement of the $(7,3,1)$ planar difference
set), the only known examples are $(11,5,2)$, $(16,6,2)$ and $(37,9,2)$.
In \cite{gordon2022difference}, it is shown that no biplanes exist in any
abelian group of order up to $10^{10}$, with six possible exceptions.

The $(37,9,2)$ biplane is an interesting case.  It is of 
$B$ type, the quartic residues mod 37, so by
Theorem~\ref{thm_add}, adding zero gives an ADS.  However, 
its sumset is smaller than $v-1$.  In fact, the elements missing from
the sumset are $0$ and $2d$ for each quartic residue $d$, so 
removing {\em any} element from the difference set gives an ADS.

The author maintains a database of abelian difference sets
\cite{ljdsr}.  An exhaustive search of the database, trying to add or
remove all possible elements from each one, only came up with the
new cases shown in
Table~\ref{tab:sporadic} 
(all of them adding an
element to a difference set).
All of these are in 2-groups, and all of the families given above are
either planar difference sets or residue difference sets (with or
without 0).  It has been conjectured that no other residue difference
sets (see, for example, \cite{byard2011lam}).
This leads the author to believe:

\begin{conjecture}
  For $G$ cyclic or $v$ odd, no other difference sets form an ADS by
  adding or removing an element.
\end{conjecture}

There is no firm basis for extending this conjecture to all groups.
On
the one hand, as $\lambda$ gets bigger, the probability of the sumset
not hitting all the elements of $G$ without a causal reason becomes
very small.
On the other, there are many undiscovered difference sets
that might work.
Omar AbuGhneim
in \cite{AbuGhneim2016} finds all (64,28,12) difference sets in {\em all}
groups of order 64, including 6656 in abelian groups, leading to most
of the examples in Table~\ref{tab:sporadic}.
No such computation has been done for difference sets in
larger 2-groups.

\begin{table}
\begin{center}
\begin{tabular}{|c|c|c|c|}
\hline
$(v,k,\lambda)$ & $G$ & $D$ & $d$  \\ \hline
$(16,6,2)$  & $\Znum_4 \times \Znum_4$ & (0, 0), (1, 0), (2, 0), (0, 1), (3, 2), (0, 3) & (1,1) \\ \hline
$(64,28,12)$  & $\Znum_8 \times \Znum_8$ & 
(0, 0), (1, 0), (0, 1), (2, 0), (0, 2), (4, 0),  & (3,1) \\
&& (0, 4), (1, 1), (3, 0), (1, 2), (1, 4), (0, 3), & \\
&& (4, 1), (4, 4), (3, 4), (1, 6), (2, 3), (2, 5), & \\
&& (4, 3), (6, 4), (4, 6), (3, 3), (5, 5), (7, 2), & \\
&& (6, 3), (6, 5), (7, 6), (7, 7) & \\ \hline
$(64,28,12)$  & $\Znum_4 \times \Znum_4 \times \Znum_4$ &
(0, 0, 0), (1, 0, 0), (0, 1, 0), (0, 0, 1), & (0,1,1)  \\
&& (2, 0, 0), (0, 2, 0), (1, 1, 0), (1, 0, 1), & \\
&& (1, 2, 0), (1, 0, 2), (2, 0, 1), (0, 0, 3), & \\
&& (2, 2, 0), (1, 1, 1), (3, 1, 0), (1, 2, 1), & \\
&& (1, 2, 2), (2, 3, 0), (2, 1, 2), (0, 3, 2), & \\
&& (2, 0, 3), (1, 1, 3), (1, 3, 2), (3, 0, 3), & \\
&& (3, 3, 1), (3, 3, 2), (3, 2, 3), (3, 3, 3)  & \\ \hline
$(64,28,12)$  & $\Znum_4 \times \Znum_4 \times \Znum_4$ &
(0, 0, 0), (1, 0, 0), (0, 1, 0), (0, 0, 1), & (1,1,1) \\
&& (2, 0, 0), (0, 2, 0), (1, 1, 0), (1, 0, 1), & \\
&& (3, 0, 0), (1, 0, 2), (0, 1, 1), (0, 1, 2), & \\
&& (2, 0, 1), (2, 2, 0), (1, 3, 0), (3, 0, 2), & \\
&& (0, 3, 1), (0, 1, 3), (2, 3, 0), (0, 2, 3), & \\
&& (3, 1, 2), (3, 2, 1), (3, 0, 3), (1, 2, 3), & \\
&& (0, 3, 3), (2, 3, 2), (2, 2, 3), (3, 3, 2) & \\ \hline

$(64,28,12)$  & $\Znum_4 \times \Znum_4 \times \Znum_4$ &
(0, 0, 0), (1, 0, 0), (0, 1, 0), (0, 0, 1), & (1,1,1) \\
&& (2, 0, 0), (0, 2, 0), (1, 1, 0), (1, 0, 1), & \\
&& (3, 0, 0), (1, 0, 2), (0, 1, 1), (2, 1, 0), & \\
&& (0, 2, 1), (2, 2, 0), (1, 0, 3), (3, 0, 2), & \\
&& (0, 3, 1), (0, 1, 3), (0, 3, 2), (2, 0, 3), & \\
&& (3, 3, 0), (3, 1, 2), (1, 3, 2), (3, 2, 1), & \\
&& (0, 3, 3), (2, 3, 2), (2, 2, 3), (3, 2, 3) & \\ \hline

$(64,28,12)$  & $\Znum_4 \times \Znum_4 \times \Znum_4$ &
(0, 0, 0), (1, 0, 0), (0, 1, 0), (0, 0, 1), & (1,1,1) \\
&& (2, 0, 0), (0, 2, 0), (1, 1, 0), (1, 0, 1), & \\
&& (3, 0, 0), (1, 0, 2), (0, 1, 1), (2, 1, 0), & \\
&& (0, 2, 1), (0, 0, 3), (2, 2, 0), (1, 0, 3), & \\
&& (3, 0, 2), (0, 3, 1), (0, 3, 2), (0, 2, 3), & \\
&& (3, 3, 0), (3, 1, 2), (1, 3, 2), (3, 2, 1), & \\
&& (2, 1, 3), (2, 3, 2), (3, 2, 3), (2, 3, 3) & \\ \hline

$(64,28,12)$  & $\Znum_4 \times \Znum_4 \times \Znum_4$ &
(0, 0, 0), (1, 0, 0), (0, 1, 0), (0, 0, 1), & (1,1,1) \\
&& (2, 0, 0), (0, 2, 0), (1, 1, 0), (1, 0, 1), & \\
&& (3, 0, 0), (1, 0, 2), (0, 1, 1), (2, 1, 0), & \\
&& (0, 2, 1), (0, 0, 3), (2, 2, 0), (1, 3, 0), & \\
&& (1, 0, 3), (3, 0, 2), (0, 3, 2), (0, 2, 3), & \\
&& (3, 1, 2), (3, 2, 1), (2, 3, 1), (2, 1, 3), & \\
&& (0, 3, 3), (2, 3, 2), (3, 3, 2), (3, 2, 3) & \\ \hline

\end{tabular}\label{tab:sporadic}
  \caption{Sporadic difference sets and added elements forming an ADS}
\end{center}
  \end{table}







\section{Modular Golomb Rulers}\label{sec:mgr}


As was stated in the introduction, 
a $(v,k,0,t)$-ADS is also known as a $(v,k)$-{\em modular Golomb
ruler} (MGR).
In \cite{bs2021} 
a number
of necessary conditions are shown for $(v,k)$-MGR to exist for various $v$ and
$k$.

The case $v = k^2-k+2$, equivalent to a $(v,k,0,1)$-ADS, is of
particular interest.  In that paper they observe:

\begin{theorem}
  Let $G = \Znum_{k^2-k+2}$, and $N$ be the subgroup of $G$ of order 2.
  A $(k^2-k+2,k)$-MGR is equivalent to a relative
  $((k^2-k+2)/2,2,k,1)$-difference set in $G$ relative to $N$.
\end{theorem}

Ryser \cite{ryser1973variants} called cyclic relative difference sets
with $n=2$ {\em near difference sets of type 1}.  His Theorem 2.1
shows

\begin{theorem}
  Let $D=\{d_1,\ldots,d_k\}$ be a cyclic relative $(v,2,k,\lambda)$
  difference set.  Then $v$ even implies that $k-2\lambda$
  is a perfect square, and $v$ odd implies that $k$ is a perfect square.
  Moreover, there exists a $(v,k,2\lambda)$-difference set.
\end{theorem}

Combining these results, we have that a $(v,k,0,1)$-ADS (and so the
corresponding $(k^2-k+2,k)$-MGR) exists only if a cyclic $(v,k,2)$-difference
set exists.

As mentioned in the previous section, these biplanes are very rare.
From the trivial $(2,2,2)$-  and $(4,3,2)$-difference sets we get
$(4,2)$- and $(8,3)$-MGRs, and the $(14,4)$-MGR from the $(7,4,2)$-difference set.
The $(11,5,2)$- and $(37,9,2)$-difference sets
do not lift to relative difference sets \cite{lam}.
The six exceptional cases in Table 3 of \cite{gordon2022difference})
all have $v$ odd and $k$ not a square, so there
are no other $(k^2-k+2,k)$-MGRs with $k \leq 10^{10}$.

Furthermore, Arasu et al.  \cite{ajmp} conjectured that the only
$(m,n,k,\lambda)$-relative difference
sets with $n=2$ are lifts of difference sets with the parameters of
complements of classical Singer difference sets.  That
conjecture would imply that there are no other such MGR.


Buratti and Stinson \cite{bs2021}
did exhaustive searches for $(v,k)$-MGR for all 
$k \leq 11$.  They used the following Lemma to bound the search.

\begin{lemma}
  Suppose there is a Golomb ruler of order $k$ and length $L$.  Then
  there is a $(v,k)$-MGR for all $v \geq 2L+1$.
\end{lemma}

Lengths of optimal Golomb rulers for small $k$ are given in, for
example, \cite{shearer200619difference}.

%

Similarly to \cite{bs2021}, we performed exhaustive backtracking
searches for $(v,k)$-MGR for given $k$ up to the Golomb
ruler bound.  They define
$$
\mgr{k} = \{ v : {\rm there \ exists \ a\  } (v,k){\rm -MGR}\},
$$
and determine $\mgr{k}$ for $k \leq 11$.

We used the canonicity test of Haanp{\"a}{\"a}
\cite{haanpaa2004constructing}, at each step, testing whether the
current set is lexicographically least under mappings of the form
$f(x)=ax+b \pmod v$, where $a,b \in \Znum_v$ and $\gcd(a,v)=1$, and
backtracking if not.
With this speedup 
we were able to confirm their results in a few minutes, and
complete several more cases:

\begin{itemize}
\item $\mgr{12} = \{133\} \cup \{156\} \cup \{158,159\} \cup  \{v: v \geq 161\}$
\item $\mgr{13} = \{168\} \cup \{183\}  \cup  \{v: v \geq 193\} $
\item $\mgr{14} = \{183\}   \cup  \{v: v \geq 225\} $
\item $\mgr{15} = \{255\}     \cup \{v: v \geq 267\} $
\end{itemize}

\section{Computational Results for general ADS}\label{sec:comp}

In \cite{zlz2006}, Zhang, Lei
and Zhang use some necessary conditions and extensive computations
to investigate ADS existence.
They provide a table of the parameters with $v \leq 50$ for which no
ADS exist.
There were a few errors in the table; the missed ADS are listed in
Table~II.

Figure~1
shows the results of exhaustive searches for most
$(v,k)$ with $v \leq 63$.  The canonicity tests used above for MGRs
were too costly in these cases, and did not eliminate enough of the
search.  Instead, the Gray code for fixed-density necklaces from
\cite{sawada2013gray} was used.  A fixed-density necklace is a
representative of a $k$-subset of a $v$-set, where every cyclic
rotation of a given set is represented exactly once.

Using the Gray code for necklaces avoided repeating cyclic shifts of a
set, although multiples mod $v$ were repeated.  However, the Gray code
is constant amortized time (CAT), so only a small constant time was
necessary to step from one necklace to the next.  Since it was a Gray
code, each necklace differs from the previous one by only one or two
swaps, so maintaining the autocorrelation sequence (and keeping track
of whether a set is an ADS) is very efficient.  

There is a definite pattern to the parameters where an ADS does not
exist.  Define
$$
\hatt = \min(t,v-1-t).
$$
Then $\hatt$ is small when
the number of differences appearing $\lambda$ and $\lambda+1$ times
are very unbalanced, and close to $v/2$ when they are relatively
balanced.  For a given column (i.e. fixed value of $k$), $\hatt$ will
be close to zero near the point where
$$
\lambda = \left\lfloor \frac{k(k-1)}{v-1} \right\rfloor
$$
changes.

Empirically, that is precisely where an ADS is less likely to exist.  For a given
$k$, increasing $v$ by one increases $\hatt$ by $\lambda$ until it
reaches a maximum near $v/2$, then decreasing until the next time
$\lambda$ changes, finally settling to $k(k-1)$ for $\lambda=0$.  The
faster change may explain why the parameters without an ADS form long
vertical lines on the left (where $\lambda$ is small), and shorter
lines towards the right.

We do not have a theorem, or even a good heuristic, to explain why
this would be so.
There are exceptions; many difference sets (with
$\hatt=0$) do exist, as do occasional ones with small but positive
$\hatt$, e.g. the $(56,17,4,3)$-ADS.  The largest $\hatt$ in the table
for which an
ADS does not exist is 12, for $(50,18,6,37)$ and $(53,17,5,40)$.

Finally, we note that
\cite{adhkm2001} ends with ``It is an open question whether
$(v,\frac{v-1}{2},\lambda,t)$ almost difference sets exist for all odd $v$'',
and this has been repeated in other places, such as \cite{ding}.  This
should add the condition that $t \neq v-1$, since it is well known that
$(v,\frac{v-1}{2},\lambda)$-difference sets do not exist for
$(39,19,9)$, $(51,25,12)$ because of the Mann test (see Theorem
VI.6.2 of \cite{bjl}), and many other
such parameters.  The first open case is $(61,30,14,30)$.

\begin{figure}\label{fig:tab}
\input{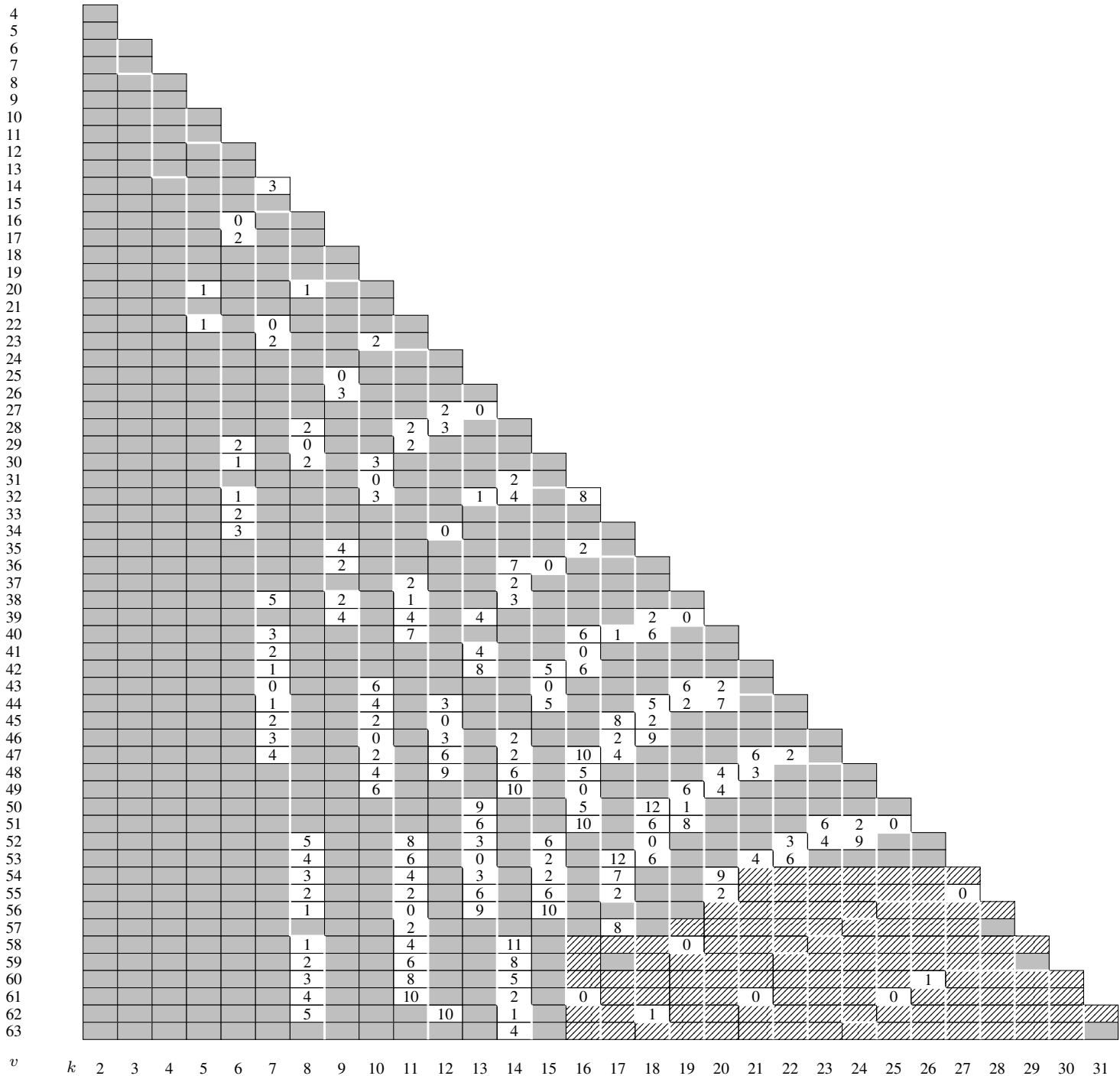}
  \caption{$(v,k)$-ADS search results: dark squares exist, white squares
    (which have the value of $\hatt$ inset) do not, dashed squares took
    too long to complete the exhaust. Thick black lines divide regions with
    different $\lambda$ values.}
\end{figure}

\begin{table}
\begin{center}
\begin{tabular}{|c|c|c|c|c|}
\hline
$v$ & $k$ & $\lambda$ & $t$ & ADS  \\ \hline
39 & 17 & 7  & 32 & $\{1,2,3,5,9,13,16,19,21,22,24,26,27,28,31,32,33\}$  \\
48 & 17 & 5  & 10 &  $\{1,2,3,5,7,9,10,16,17,18,21,24,27,29,30,34,39\}$ \\
48 & 22 & 9  & 8 &  $\{1,2,3,5,6,13,19,20,21,24,25,27,28,29,31,33,34,37,39,40,42,44\}$ \\
48 & 23 & 10  & 11 & $\{1,2,3,4,6,7,9,10,11,15,17,18,20,22,24,25,28,29,30,34,37,40,41\}$  \\
50 & 20 & 7  & 12 &  $\{1,2,3,5,7,8,10,12,17,18,20,21,24,25,28,29,31,37,42,43\}$ \\

 \hline

\end{tabular}
\caption{ADS missed in \cite{zlz2006}}
\label{tab:errata}
\end{center}
  \end{table}

\printbibliography

\appendix[Octic almost difference sets]

To show when octic residues, with or without $0$, form an almost
difference set, we must use cyclotomy.  See
\cite{berndt1998gauss} or \cite{lehmer1955number} 
for background on cyclotomy, and
\cite{hall1998combinatorial} for application of cyclotomy to difference
sets.

For an integer $e>1$ and prime power $q=ef+1$, let $g$ be a generator
for $GF(q)^*$, and 
define the cyclotomic classes $C_i^{(q,e)} = C_i$, for
$i=0,1,\ldots,e-1$ by
$$
C_i = \{ g^{es+i}, s = 0,1,\ldots,f\}.
$$
for $q$ prime $C_0$ is the set of $e$th power residues modulo $q$.

The {\em cyclotomic numbers of order $e$} are:

$$
(i,j)_e = |\{(x,y): x \in C_i, y \in C_j, x+1=y\}|
$$
for $i,j=0,1,\ldots,e-1$.

Computing cyclotomic numbers is in general difficult, but it has been
done for various small values of $e$.  For $e=8$  \cite{lehmer1955number}, let $p = 8f+1$ be a
prime with the representation
$$p = x^2 + 4y^2 = a^2 + 2b^2,$$
where $x \equiv a \equiv 1 \pmod 4$.

if $p\equiv 1 \pmod
{16}$, the matrix of cyclotomic numbers is:
$$
\left[
\begin{array}{cccccccc}
A & B & C & D & E & F & G & H \\
B & H & I & J & K & L & M & I \\
C & I & G & M & N & O & N & J \\
D & J & M & F & L & O & O & K \\
E & K & N & L & E & K & N & L \\
F & L & O & O & K & D & J & M \\
G & M & N & O & N & J & C & I \\
H & I & J & K & L & M & I & B \\
\end{array} \right]
$$
where the coefficients are given by 
  \begin{eqnarray*}
    64A & = & p -23-18x-24a, \\
    64B & = & p -7+2x+4a+16y+16b, \\
    64C & = & p -7+6x+16y, \\
    64D & = & p -7+2x+4a-16y+16b, \\
    64E & = & p -7 -2x+8a, \\
    64F & = & p -7+2x+4a+16y-16b, \\
    64G & = & p -7+6x -16y, \\
    64H & = & p -7+2x+4a-16y-16b, \\
    64I & = & p +1+2x - 4a, \\
    64J & = & p +1 -6x + 4a, \\
    64K & = & p +1 +2x - 4a, \\
    64L & = & p +1  +2x - 4a, \\
    64M & = & p +1  -6x + 4a, \\
    64N & = & p +1  -2x, \\
    64O & = & p +1  +2x - 4a \\
  \end{eqnarray*}
for 2 a quartic residue $\pmod p$, and

  \begin{eqnarray*}
   64A  & = &  p  - 23 +6x, \\
   64B  & = &  p  -7 + 2x + 4a, \\
   64C  & = &  p  -7 - 2x -8a-16y, \\
   64D  & = &  p  -7 + 2x + 4a, \\
   64E  & = &  p  -7 -10x, \\
   64F  & = &  p  -7 + 2x + 4a, \\
   64G  & = &  p  -7 - 2x -8a + 16y, \\
   64H  & = &  p  -7 + 2x + 4a, \\
   64I  & = &  p  +1 - 6x + 4a, \\
   64J  & = &  p  +1 + 2x - 4a - 16b, \\
   64K  & = &  p  + 1 + 2x - 4a+16y, \\
   64L  & = &  p  + 1 +2x - 4a-16y, \\
   64M  & = &  p  + 1 + 2x - 4a+16b, \\
   64N  & = &  p  +1 + 6x+8a, \\
   64O  & = &  p  + 1 - 6x + 4a \\
  \end{eqnarray*}
for 2 not a quartic residue.

For $p \equiv 9 \pmod {16}$, the cyclotomic numbers are:
$$
\left[
\begin{array}{cccccccc}
A & B & C & D & E & F & G & H \\
I & J & K & L & F & D & L & M \\
N & O & N & M & G & L & C & K \\
J & O & O & I & H & M & K & B \\
A & I & N & J & A & I & N & J \\
I & H & M & K & B & J & O & O \\
N & M & G & L & C & K & N & O \\
J & K & L & F & D & L & M & I \\
\end{array} \right]
$$
where
  \begin{eqnarray*}
    64A & = & p -15-2x, \\
    64B & = & p +1+2x-4a+16y, \\
    64C & = & p + 1+6x+8a-16y, \\
    64D & = & p + 1+2x-4a-16y, \\
    64E & = & p + 1 - 18x, \\
    64F & = & p + 1 +2x - 4a + 16y, \\
    64G & = & p + 1  +6x +8a + 16y, \\
    64H & = & p + 1  +2x - 4a - 16y, \\
    64I & = & p -7  +2x + 4a, \\
    64J & = & p -7  +2x + 4a, \\
    64K & = & p + 1  -6x + 4a + 16b, \\
    64L & = & p + 1  +2x - 4a, \\
    64M & = & p + 1  -6x + 4a - 16b, \\
    64N & = & p - 7  -2x - 8a, \\
    64O & = & p + 1  +2x - 4a \\
  \end{eqnarray*}
for 2 a quartic residue $\pmod p$, and

  \begin{eqnarray*}
   64A  & = &  p  - 15 - 10x - 8a, \\
   64B  & = &  p  + 1 + 2x - 4a - 16b, \\
   64C  & = &  p  + 1 - 2x + 16y, \\
   64D  & = &  p  + 1 + 2x - 4a - 16b, \\
   64E  & = &  p  + 1 + 6x + 24a, \\
   64F  & = &  p  + 1 + 2x - 4a + 16b, \\
   64G  & = &  p  + 1 - 2x - 16y, \\
   64H  & = &  p  + 1 + 2x - 4a + 16b, \\
   64I  & = &  p  - 7 + 2x + 4a + 16y, \\
   64J  & = &  p  - 7 + 2x + 4a - 16y, \\
   64K  & = &  p  + 1 + 2x - 4a, \\
   64L  & = &  p  + 1 - 6x + 4a, \\
   64M  & = &  p  + 1 + 2x - 4a, \\
   64N  & = &  p  - 7 + 6x, \\
   64O  & = &  p  + 1 - 6x + 4a \\
  \end{eqnarray*}
for 2 not a quartic residue.

Lehmer (\cite{lehmer}, Theorem 1) shows that for the octic residues to
form a
$(p,k,\lambda)$-difference set, we must have
$$(i,0) = \lambda, \ i = 0,1,\ldots 7.$$
This can only happen for $p \equiv 9 \pmod {16}$, since in the other
case Lemma 1 of that paper shows that one of the $(i,0)$ is odd, and
the rest even.  For ADS we will need to consider both cases.

For an octic difference set we need $A = I = N = J$, and solving the
equations gives $a=1$, $x= -3$, so
\begin{equation}\label{octrep}
p = 9 + 4y^2 = 1 + 2 b^2.  
\end{equation}

How do we know if there are any such numbers?   We may rewrite
(\ref{octrep}) as
\begin{equation}\label{normeq}
b^2 - 2 y^2 = 4.  
\end{equation}

Note that both $b$ and $y$ must be even.
(\ref{normeq}) says that in the field
$K = \Qnum(\sqrt{2})$, the number $b+\sqrt{2} y$ has norm 4
(see, for example, \cite{cohn1980advanced} for facts about quadratic number fields).
$K$ has fundamental unit $u = 1+\sqrt{2}$ with norm $-1$.
The prime $2$ ramifies in $K$, so the elements of norm 4 are of the
form $2 u^{2i}$.  For example, $2 u^2 = 6 + 4 \sqrt{2}$, and
$73 =  3^2 + 4 \cdot 4^2 = 1^2 + 2 \cdot 6^2$, so the octic residues modulo 73
form a difference set.  The next one is
$2 u^{14}=  228486 + 161564 \sqrt{2}$, giving
$104411704393 = 3^2 + 4 \cdot 161564^2 = 1^2 + 2 \cdot 228486^2$
It is an open question whether there are an infinite number of these,
since
we do not know if this sequence yields primes infinitely often.
A naive heuristic argument suggests that it does, but very slowly,
with the number of such primes up
to $x$ being ${\cal{O}}(\log \log x)$.

Consider the case of almost difference sets for $p \equiv 9 \pmod {16}$.
Now we need some subset of $\{A,I,N,J\}$ being equal to $\lambda$, and
the rest $\lambda+1$.  Checking all the possibilities gives the
solutions in Table~\ref{tab:octic}.
The results for ADS with $t=(v-1)/2$, such as
Theorems 6.6.5 and 6.6.6 of \cite{cusick2004stream},
gave lines 3 and 4 in the table, 
which were shown not to have any solutions in \cite{qi2016nonexistence}.

\begin{table}
\begin{center}
\begin{tabular}{|c|c|c|c|}
\hline
$a$ & $x$ & norm($b^2-2y^2$) & factors  \\ \hline
-7 & 21  & 196  & $2^2 \cdot 7^2$  \\
9 & -11  & 20  & $2^2 \cdot 5$  \\
1 & 13  & 84  & $2^2 \cdot 3 \cdot 7$ \\
1 & -19  & 180  & $2^2 \cdot 3^2 \cdot 5$  \\
-7 & 5  & -12  &  $-2^2 \cdot 3$ \\
9 & -27  & 324  &  $2^2 \cdot 3^4$ \\
 \hline
\end{tabular}
  \caption{Solutions to Type $O$ ADS cyclotomic equations, $p \equiv 9
    \pmod{16}$, 2 a quartic residue}
  \label{tab:octic}
\end{center}
  \end{table}

3 and 5 are both inert in $K$, while $7 = (1+2\sqrt{2}) (-1+2\sqrt{2})$.
Thus there are no elements with norm 3 or 5, meaning that
none of the middle lines can have solutions.  For the last line, the
only element of norm $2^2 3^4$ is 18, which will lead to $b$ and $y$
always having factors of 3.  Since $a$ and $x$ do as well, none of the
values will be prime.

For the first line,
we may take $b+\sqrt{2}y$ to be either
$2(1-2\sqrt{2})u^{2i}$ or  $2(1+2\sqrt{2})u^{2i}$.  Both lead to solutions, the
first starting with  $p=26041$, the second with $p=660279756217$.

Solving the equations for 2 a quartic nonresidue, the only
solution we get is $p=41$.

For type $O_0$, the addition of $\{0\}$ will increase whichever
$(i,0)_8$ contain $1$ and $-1$.  $1$ is in $C_0$, so $A=(0,0)_8$ will
increase by 1.  If $f$ is odd, then $-1 \equiv g^{4f}$ is in $C_4$, so
$(4,0)_8$ (which is also equal to $A$) increases by 1.  If $f$ is even,
then $-1 \in C_0$, so $A$ increases by 2.

Going through all the possibilities with these new equations, we get
the solutions in Table~\ref{tab:oplus}.

\begin{table}
\begin{center}
\begin{tabular}{|c|c|c|c|}
\hline
$a$ & $x$ & norm($b^2-2y^2$) & factors  \\ \hline
-15 & 45  & 900  & $2^2 \cdot 3^2 \cdot 5^2$  \\
1 & 13  & 84  & $2^2 \cdot 3 \cdot 7$ \\
-7 & 37  & 660  & $2^2 \cdot 3  \cdot 5  \cdot 11$  \\
-7 & 5  & -12  &  $-2^2 \cdot 3$ \\
-15 & 29  & 308 & $2^2 \cdot 7 \cdot 11$  \\
1  & -3  & 4  &  $2^2$ \\
 \hline
\end{tabular}
  \caption{Solutions to $O_0$ ADS cyclotomic equations, $p \equiv 9
    \pmod{16}$, 2 a quartic residue}
  \label{tab:oplus}
\end{center}
  \end{table}

11 is also inert in $K$, so there are no solutions for lines with a
single power of 3, 5 or 11.  The first line has no prime solutions
because of a common factor of 5.
The last line, $a=1$, $x=-3$, corresponds to the ADS
gotten in Theorem~\ref{thm_octic} by adding $\{0\}$ to a type O
difference set.

Solving the same equations for 2 a quartic nonresidue, the only
solution we get is again $p=41$.  So there is a Type $O$
$(41,5,0,20)$-ADS and a Type $O_0$
$(41,6,0,10)$-ADS.

In the same way, all the possibilities with $p \equiv 1 \pmod {16}$
were tried, for 2 a quartic residue or not.  The only solution was a $(17,2,0,14)$-ADS.

Thus, aside from 17 and 41, all Type $O$ and $O_0$ ADS arise from
Theorem~\ref{thm_octic}, and 
we have the following classification of octic residue almost
difference sets:

\begin{theorem}\label{thm:typeO}
  A Type $O$ ADS exists in $GF(p)$ if and only if $p$ is 17, 41, or a prime such that
$p = 8 t^2 + 49 = 64 u^2 + 441$, where $t$ is odd and $u$ is even.
\end{theorem}

\begin{theorem}\label{thm:typeO0}
  A Type $O_0$ ADS exists in $GF(p)$ if and only if $p$ is 41 or a prime such that
$p = 8 t^2 + 1 = 64 u^2 + 9$, where $t$ and $u$ are odd.
\end{theorem}


\end{document}